\documentclass[12pt, reqno]{amsart}
\usepackage{amsmath, amsthm, amscd, amsfonts, amssymb, graphicx, color}
\usepackage{setspace}
\usepackage{mathrsfs}
\usepackage{multicol}
\usepackage[bookmarksnumbered, colorlinks, plainpages]{hyperref}
\hypersetup{colorlinks=true,linkcolor=red, anchorcolor=green, citecolor=cyan, urlcolor=red, filecolor=magenta, pdftoolbar=true}

\newtheorem{theorem}{Theorem}[section]
\newtheorem{lemma}[theorem]{Lemma}

\newtheorem{corollary}[theorem]{Corollary}
\theoremstyle{definition}

\theoremstyle{remark}

\numberwithin{equation}{section}

\newcommand{\NN}{\mathbb{N}}

\newcommand{\CC}{\mathbb {C}}

\begin{document}
\setcounter{page}{1}
\title[ Topological and dynamical properties of  composition operators]{ Topological and dynamical properties of composition operators }
\author[Tesfa  Mengestie]{Tesfa  Mengestie }
\address{ Mathematics Section \\
Western Norway University of Applied Sciences\\
Klingenbergvegen 8, N-5414 Stord, Norway}
\email{Tesfa.Mengestie@hvl.no}
\author [Werkaferahu Seyoum]{Werkaferahu Seyoum}
\address{Department of Mathematics,
Addis Ababa University, Ethiopia}
\email{Werkaferahus@gmail.com}
\thanks{The research of the second  author  is partially supported by ISP project, AAU}
\begin{abstract}
We study various  properties of   composition operators acting between generalized Fock spaces  $\mathcal{F}_\varphi^p$  and $\mathcal{F}_\varphi^q$ with weight functions $\varphi$ grow faster than the classical Gaussian weight function  $\frac{1}{2}|z|^2$ and satisfy some mild smoothness conditions.  We have shown  that if  $p\neq q,$  then the composition operator   $C_\psi: \mathcal{F}_\varphi^p \to  \mathcal{F}_\varphi^q $ is bounded if and only if it is compact. This result  shows a significance difference with the analogous  result for the case  when  $C_\psi$ acts  between the classical  Fock spaces or generalized Fock spaces  where the weight functions grow  slower than   the Gaussian weight  function. We further  described  the Schatten $\mathcal{S}_p(\mathcal{F}_\varphi^2)$ class, normal, unitary, cyclic and supercyclic composition operators. As an application, we characterized the compact differences,  the isolated and essentially isolated  points, and  connected components of the space of the  operators under  the operator norm topology.
\end{abstract}
\maketitle

\section{Introduction}
For given holomorphic mappings $\psi$  and $f$ on the complex plane $\CC$, we define the composition operator  induced by $\psi$   as
$C_\psi f= f(\psi).$
 Composition operators    have been extensively studied on various  spaces of holomorphic functions  over several  settings in the past  many  years. It is rather  difficult   to give a comprehensive list of related   works on the subject now.  For an overview in the framework of  Fock  spaces, which we are  interested in, one may consult the materials for example  in \cite{CMA, CCK, GK, LGR, TM5} and  the references therein. On the other hand, over the unit  disc of the complex plane or the unit ball  in $\CC^n$, the monographs in \cite{CO,JA, SH} provide a  comprehensive expositions specially on the early developments of the area.
  The study of composition operators has continued to gain voluminous interest partly because it finds itself at the interface of both  operator and function theories.

 In the classical Fock spaces setting, the  boundedness and compactness properties of composition operators were studied  for example in \cite{CMA, CCK, TM, TM0}. On the other hand, when the weight function generating the generalized Fock spaces grows slower than the classical Gaussian weight function, a thorough  look into the proof of Proposition~2.1 of \cite{TM5} shows that the forms of the symbol $\psi$ inducing  bounded and compact $C_\psi$ are  just like that of the classical case. A similar result can be also read in \cite{CCK,GK,TM4} where the form of the operator is discussed on Fock--Sobolev spaces which are typical examples of generalized Fock spaces with weight function growing  slower than the Gaussian weight function.   A natural question is what happens to these properties when the weight function grows faster than the Gaussian function.  The aim of this work is taking further the study of  the  operators on such spaces, and answer these and other related topological and dynamical  questions.  It turns out that while  the dynamical structures of the operators behave like that of the classical setting, the faster growth of the weight function results in a poorer structure in the forms of the  symbols inducing bounded composition operators acting between two  generalized Fock spaces.

We begin by setting  the growth  and smoothness conditions for the  generating weight function.  Let $\varphi:[0, \infty) \to [0, \infty)$ be a twice continuously differentiable function. We extend $\varphi$ to the whole complex plane by setting $\varphi(z)= \varphi(|z|)$.    We further assume that its  Laplacian,
 $\Delta \varphi $,   is positive and set $ \tau(z)\overset{\footnotemark}{\simeq} (\Delta\varphi(z))^{-1/2}$ when $ |z|\geq 1,$ and  $ \tau(z)\simeq 1$ whenever $ 0\leq|z|<1$,\footnotetext{The  notation $U(z)\lesssim V(z)$ (or
equivalently $V(z)\gtrsim U(z)$) means that there is a constant
$C$ such that $U(z)\leq CV(z)$ holds for all $z$ in the set of a
question. We write $U(z)\simeq V(z)$ if both $U(z)\lesssim V(z)$
and $V(z)\lesssim U(z)$.}  where $\tau$ is a radial differentiable function satisfying the admissibility conditions
  \begin{align}
  \label{assumption}
  \lim_{r \to \infty} \tau(r)= 0\ \ \text{and} \ \   \lim_{r\to \infty} \tau'(r)=0,
  \end{align} and there exists a constant $C>0$ such that $\tau(r)r^C$ increases for large $r$ or
 \begin{align*}
 \lim_{r\to \infty} \tau'(r)\log\frac{1}{\tau(r)}=0.
 \end{align*}
We may note   that
 there are many  concrete examples of weight functions $\varphi$ that satisfy the above smoothness and admissibility   conditions. The  power functions $\varphi_\alpha(r)= r^\alpha, \ \alpha>2$,   the exponential functions such as $\varphi_\beta(r)= e^{\beta r},\ \ \beta >0 $,  and  the  supper exponential functions $ \varphi(r)= e^{e^{ r}}$  are all typical examples  of such  weight functions.

Having set forth  the conditions on $\varphi$, we may now define  the associated generalized  Fock spaces $\mathcal{F}_\varphi^p$  as  spaces  consisting of all entire functions $f$ for which
\begin{align*}
\|f\|_{p}^p= \int_{\CC} |f(z)|^p
e^{-p\varphi(z)} dA(z) <\infty,
\end{align*} where   $dA$ denotes the
usual Lebesgue area  measure on $\CC$.
These spaces have been studied in various contexts in the past years for instance  in \cite{Borch, Olivia, TM3}.

It has been  known that the Laplacian $\Delta\varphi$ of the weight function $\varphi$ plays a significant role  in the study of various operators on generalized Fock spaces. Now it   is found that  the structure of the symbol $\psi $ inducing  a bounded map $C_\psi$ is determined based on the growth of  $\Delta\varphi$. More specifically, if  $C_\psi$ acts between two different generalized Fock spaces and $ \Delta\varphi(z) \to \infty$ as $|z| \to \infty$, then one of  our main result shows that  $C_\psi$ experiences a poorer boundedness structure than the case when the Laplacian is uniformly bounded. As a consequence of the growth of the Laplacian, we will see that the space of all  bounded  composition operators $ C_\psi:  \mathcal{F}_{\varphi}^p \to \mathcal{F}_\varphi^q, \ \ \ p\neq q$  is connected  under the operator norm topology. On the other hand, as will be seen in Subsection~\ref{dynamical}, the dynamical properties such as cyclicity, hypercyclicity and   supercyclicity show no dependency on the growth of the Laplacian as in the classical setting.
\section{main results}\label{mainresults}
In this section, we  state the  main results  of this note and  defer  their proofs for  Section \ref{proofs}.
Our first  main result  describes the bounded and compact composition operators acting between the  spaces.
\begin{theorem}\label{thm1} Let $0<p,  q < \infty$ and $\psi $ be a nonconstant holomorphic map on the complex plane  $\CC$. If
\begin{enumerate}
\item  $p\neq q,$ then the following statements are equivalent.
\begin{enumerate}
\item $C_\psi:  \mathcal{F}_{\varphi}^p \to \mathcal{F}_\varphi^q$ is bounded;
\item $C_\psi:  \mathcal{F}_{\varphi}^p  \to \mathcal{F}_\varphi^q$ is compact;
\item $ \psi(z) = az+b$ for some complex numbers $a$ and $b$ such that $|a|<1$.
\end{enumerate}
\item   $p=q$, then  $C_\psi:  \mathcal{F}_{\varphi}^p \to \mathcal{F}_\varphi^q$ is
\begin{enumerate}
\item bounded  if and only if $ \psi(z) = az+b$ for some complex numbers $a$ and $b$ such that $|a|\leq 1$,  and $b=0$ whenever $|a|=1$.
\item compact if and only if $ \psi(z) = az+b$ for some complex numbers $a$ and $b$ such that $|a|<1$.
\end{enumerate}
\end{enumerate}
    \end{theorem}
   Part (i) of the result  shows a significance difference with the corresponding  result for the case  when  $C_\psi$ acts  between  the classical Fock spaces or generalized Fock spaces  where the weight function $\varphi$ grows  slower than  the  Gaussian weight function \cite{CMA,CCK,GK,TM4,TM5,TM}. On the other hand, part (ii) of the result  simply extends the classical results with the same form. It is interesting to observe the sharp contrast between the cases when $p=q$ and $p\neq q$.   Unlike the classical setting where these two cases result no different conditions, the boundedness structure gets poorer when the weight grows faster than the  Gaussian case. Furthermore, apart  from the fact that $p\neq q$,  all  of the results are independent of the size of the  exponents $p$ and $q$  in the range $0<p,q<\infty$.
\subsection{ Essential norms and Schatten class membership}\label{essential}
In this section we turn our attention  to the study of essential norms and Schatten class membership of
composition operators on the  spaces $\mathcal{F}_\psi^p$.   For a bounded linear  operator $T$ on a Banach space $\mathcal{H}$, we recall that  the essential norm $\|T\|_e$ of $T$ is the norm of its equivalence classes in the Calkin algebra. In other words,
     \begin{align*}
     \|T\|_e= \inf\{\|T-K\|;  \text{K is a compact operator}\}.
     \end{align*}
     It follows that  $\|T\|_e \leq \|T\|$ and $\|T\|_e= 0$ whenever  $T$ is a compact operator.
     Computing  the values of  the norms and essential norms   of composition operators is not an
easy task and hence, not much is known on these problems.
In this section, we will estimate these values for $C_\psi$ on the   spaces $\mathcal{F}_\varphi^p$ for $p\geq 1$.  For a noncompact operator $C_\psi$, we will indeed   show that its   essential norm $\|C_\psi\|_e$  is comparable to its operator norm $\|C_\psi\|$. In the Hilbert space setting $\mathcal{F}_\varphi^2$,  we have rather  obtained the precise values of  the norms, namely that  $\|C_\psi\|_e =\|C_\psi\| =1$.

 If  $C_{\psi}$ is  a compact operator on $\mathcal{F}_\varphi^2$, then  it admits a Schmidt decomposition, and there exist  orthonormal bases $(e_n)_{n\in\NN}$  and $(\sigma_n)_{n\in\NN}$, and a sequence of nonnegative numbers  $(\lambda_{n,\psi)})_{n\in \NN}$ with $\lambda_{(n,\psi)} \to 0$  as $n\to \infty$  such that for all $f$ in $\mathcal{F}_\varphi^2$:
\begin{align*}
 C_{\psi} f= \sum_{n=1}^\infty \lambda_{(n,\psi)} \langle f,e_n\rangle \sigma_n.
\end{align*}   The operator  $C_{\psi}$ with such a decomposition  belongs to  the
Schatten  $\mathcal{S}_p(\mathcal{F}_\varphi^2)$ class if
and only if
\begin{equation*}
\|C_{\psi}\|_{\mathcal{S}_p}^p = \sum_{n=1}^\infty |\lambda_{(n,\psi)}|^p <\infty.
\end{equation*} We refer to  \cite{BS,kz} for a  more detailed  account of the theory of Schatten classes.
It turns out that  all compact composition operators on $\mathcal{F}_\varphi^2$  are in the Schatten $S_p$ class
for all $0<p<\infty$. We may now summarize  all the above observations in the following theorem.
 \begin{theorem}\label{thm2}
Let $0<p<\infty $ and $\psi $ be a nonconstant holomorphic map on $\CC$  that induces  a bounded operator  $C_\psi$    on $\mathcal{F}_\varphi^p$. Then if  $C_\psi$  is
\begin{enumerate}
\item   not compact on $\mathcal{F}_\varphi^p$ for $p\geq 1$, then its essential norm  is comparable with  its operator norm and
        \begin{align}
    \label{estim}
    1\geq \|C_\psi\|_e \simeq \|C_\psi\|.
    \end{align}
    On the Hilbert space $\mathcal{F}_\varphi^2$ case, we have equality and
    \begin{align}
    \label{exact}
   \|C_\psi\|_e =\|C_\psi\| =1.
         \end{align}
  \item  compact on $\mathcal{F}_\varphi^2$, then it belongs to the Schatten
$\mathcal{S}_p(\mathcal{F}_\varphi^2)$ class for all $0<p<\infty$.
\end{enumerate}
 \end{theorem}
In the classical setting for $p=2$, it has been proved   \cite{CMA} that the norm and essential norm of $C_\psi$ are equal and
\begin{align}
\label{classical}
 \|C_\psi\|_e = \|C_\psi\| = 1,
\end{align} where $C_\psi$ is noncompact operator induced by $\psi(z)= az, \ |a|=1$. The proof of \eqref{classical} uses Hilbert spaces techniques  based on  an explicit expression of the reproducing kernel. In the  current setting,  an explicit expression for the kernel function is still an open problem. Yet,  we managed in circumventing this difficulty by using an asymptotic estimate of $\|K_z\|_2$ as $|z| \to \infty$ and arrive at \eqref{exact}. For $p\neq 2$, we will instead use another sequence of test functions where we only know the estimated values of the functions. It remains an interesting open problem to compute the precise values of
the estimates in \eqref{estim}.  On the other hand,  results  in \cite{TM} show that every compact composition operator acting on the classical Fock space belongs to the Schatten $\mathcal{S}_p$ class for all positive $p.$ In  sprit of this, part  (ii) of our theorem shows that the result remains valid  in generalized Fock spaces generated by  fast  growing weight functions.
\subsection{Normal and unitary composition operators} \label{normal}
In this section we characterize mappings $\psi$ which induce hyponormal, normal,  and unitary composition operators $C_\psi$ on the spaces $\mathcal{F}_\varphi^2$.
Recall that a bounded linear operator $T$  on a complex Hilbert space  $\mathcal{H}$ is  said to be hyponormal if $T^*T\geq TT^*$ where $T^*$ is the adjoint of $T$. The operator is normal if $TT^* = T^*T$, and
unitary  whenever  $TT^* = T^*T = I,$ where $I$  is the identity operator on $\mathcal{H}$. Note that  while a hyponormal operator is a generalization of a normal operator,  not all  normal operators are   unitary. Our next main result shows that three of these  operator-theoretic  properties of  $C_\psi$  acting on  the space $\mathcal{F}_\varphi^2$ are equivalently described by the same condition.
\begin{theorem}\label{thm3}
Let $\psi(z)= az+b$ induces a bounded composition operator $C_\psi$  on  $\mathcal{F}_\varphi^2$. Then
the following statements are equivalent.
\begin{enumerate}
\item
 $C_\psi$ is  hyponormal;

 \item It holds that   $|a|= 1$;

 \item  $C_\psi$ is   unitary;
 \item
 $C_\psi$ is   normal.
 \end{enumerate}
\end{theorem}
On the classical Fock space, these results were proved recently in \cite{LGR}. Our result now shows that these properties are independent of the fast growth of the inducing weight function $\varphi$.

 An interesting related property is the notion of essentially normal.  Recall that a bounded  $C_\psi$ is essentially normal if
 the commutator $[C_\psi^*, C_\psi]= C_\psi^* C_\psi-C_\psi C_\psi^*$ is compact. Then, the following is  an  immediate consequence of Theorem~\ref{thm3} and Theorem~\ref{thm1}.
 \begin{corollary}
   Let $\psi(z)= az+b$ induces a bounded composition operator $C_\psi$  on  $\mathcal{F}_\varphi^2$. Then $C_\psi$ is essentially normal.
 \end{corollary}
  By Theorem~\ref{thm1} either $|a|=1$ in which case by Theorem~\ref{thm3},  the operator is normal or $|a| <1$ and the operator becomes compact.
  Since normal and compact operators are  essentially normal,  the corollary trivially holds.
 \subsection{Dynamics of the  composition operators  on $\mathcal{F}_\varphi^p$ }\label{dynamical}
A bounded linear operator $T$ on a Banach  space $\mathcal{H}$ is said to be  cyclic  if there exists a vector $x$ in $\mathcal{H}$ such that the linear span of  its orbit under $T$,
\begin{align*}\text{Orb}(T,x)=\{ T^nx : n= 0, 1, 2, ..\},\end{align*}
is dense  in $\mathcal{H}$. Such a vector $x$ is called cyclic for the operator $T$. The operator is hypercyclic if the orbit itself is dense in $\mathcal{H}$, and supercyclic   if there exists a vector $x$ in $\mathcal{H}$ such that the projective orbit, \begin{align*}
\text{Projorb}(T,x)=\{ \lambda T^nx: \ \  \lambda \in \CC, n= 0, 1, 2, .. \},
\end{align*} is dense in  $\mathcal{H}$. Clearly any hypercyclic operator is cyclic, but the cyclic operators form a much larger class while supercyclicity is an intermediate property between the two. It is worth mentioning that if an operator $T$ has a hypercyclic vector, then each element in the orbit of such vector is also hypercyclic  which implies that a hypercyclic  operator has a dense set of hypercyclic vectors. For more information about hypercyclicity and supercyclicity, one may  consult  the  monographs by Bayart
and Matheron \cite{BM},  and by Grosse-Erdmann and Peris Manguillot \cite{Grosse}.

 There have been much interest for a long time in studying these properties partly because of their relations to  the famous invariant subspaces open problem which conjectures that every bounded linear  operator on a Banach space has a non-trivial closed invariant subspace.
On the other hand, the operator $C_\psi^n$ is itself a composition operator induced by the $n^{th}$ iterate of $\psi$,
\begin{align}
\label{dependent}
C_\psi^n
=  C_{\psi^n}, \ \ \psi^n= \underbrace{\psi \circ \psi\circ\psi \circ...\circ \psi}_{n \ \text{times}},
\end{align} which obviously makes the study of the dynamical properties  a natural subject. Furthermore, the relation in \eqref{dependent}  indicates that the  dynamical behavior of a composition operator is heavily  dependent on the dynamical properties of its inducing map $\psi$.

  In this section, we study the dynamical properties of  the composition operator on $\mathcal{F}_p$. We may first make the following simple observation, namely that  no bounded composition operator on $\mathcal{F}_\varphi^p$ can be hypercyclic.  To  notice this, set $\psi(z)= az+b$  and observe that if  $|a|<1$, then  the operator $C_\psi$ is compact and hence by Corollary 1.22 of \cite{BM}, it can not be hypercylic. On the other hand,  if $|a|=1,$ we may  deny the assertion and assume that the operator is hypercyclic with hypercyclic vector $f$.  By  extracting  a subsequence $\psi^{n_k} $ such that $\psi^{n_k} z \to az $ as $k \to \infty$, we observe that  for any univalent function $g$ in the orbit of $f$ and   applying \eqref{dependent}
 \begin{align*}
 g(z)= \lim_{k\to \infty} C_{\psi^{n_k}}f(z)= \lim_{k\to \infty} C_{\psi^{n_k}}f(z)= f(az).
 \end{align*} It follows  that $f$ itself  is a univalent function and hence its orbit contains only univalent functions which is a contradiction.
   Our next main result shows that the operator can not be supercyclic either.
\begin{theorem}\label{thm30}
Let $1\leq p <\infty$ and  $\psi(z)= az+b$ be a nonconstant map on $\CC$ that  induces a bounded composition operator $C_\psi$ on  $\mathcal{F}_\varphi^p$. Then $C_\psi$
\begin{enumerate}
\item  is cyclic on $\mathcal{F}_\varphi^p$ if and only if  $a^n \neq a$ for all $n>1$.  Furthermore, a function $h \in \mathcal{F}_\varphi^p$ with  Taylor series  expansion
 \begin{align*} h(z)=
  \sum_{n=0}^\infty a_n \bigg(z-\frac{b}{1-a}\bigg)^n
    \end{align*}
  is cyclic for $C_\psi$  if and only if $a_n\neq 0$ for all $n\in {\mathbb{Z}_+}:= \{0, 1, 2, 3, ...\} $.
\item can not   be  supercylic on $\mathcal{F}_\varphi^p$.
\end{enumerate}
 \end{theorem}
  The cyclicity and  supercyclicity  problems have not been solved in the classical Fock spaces settings either  except for the Hilbert space   case which were studied respectively in   \cite{GK} and  \cite{LGR}. As can be seen in Section \ref{proofs}, our approach, which neither uses Hilbert spaces techniques nor the fast growth property of the weight function $\varphi$,  shows that the same  result holds for  all $p$ on the classical spaces as well.
 \subsection{Connected components  and isolated points of  the space of composition operators}\label{topological}
 In the present section we consider  some topological structures of  bounded compositions operators $C_\psi: \mathcal{F}_\varphi^p\to \mathcal{F}_\varphi^q$ for all  $0<p,q<\infty$. We denote by $C(\mathcal{F}_\varphi^p, \mathcal{F}_\varphi^q)$ the space of such  operators equipped with the operator norm topology. The first natural question to pose  in this direction  is when the difference of two operators from  $C(\mathcal{F}_\varphi^p, \mathcal{F}_\varphi^q)$ becomes compact. It turns out that the difference is compact if and only if both of the operators are compact.  Another natural  point of interest is to identify the isolated and connected component of  the space  $C(\mathcal{F}_\varphi^p, \mathcal{F}_\varphi^q)$ which we give a complete characterization below.
 \begin{theorem}\label{thm4}
 \begin{enumerate}
 \item
Let $0<p,q <\infty$ and $C_\psi$ be in $C(\mathcal{F}_\varphi^p, \mathcal{F}_\varphi^q)$.
If
 \begin{enumerate}
 \item $p\neq q$, then the space  $C(\mathcal{F}_\varphi^p, \mathcal{F}_\varphi^q)$ is connected.
      \item $p= q$, then   $C_\psi $ is isolated if and only if it is not compact. In this case,  the set of all compact composition operators on $ \mathcal{F}_\varphi^p$ is a connected component of $ C(\mathcal{F}_\varphi^p, \mathcal{F}_\varphi^p)$.
 \end{enumerate}
 \item Let  $0<p <\infty$ and $C_{\psi_1}, C_{\psi_2}  \in C(\mathcal{F}_\varphi^p, \mathcal{F}_\varphi^p)$  where $\psi_1\neq \psi_2$. Then
      \begin{enumerate}
      \item $C_{\psi_1}- C_{\psi_2}$ is compact on $\mathcal{F}_\varphi^p$ if  and only if both  $C_{\psi_1}$ and  $ C_{\psi_2}$  are compact.
          \item  if  $C_{\psi_1}- C_{\psi_2}$ is compact on $\mathcal{F}_\varphi^2,$ then it belongs to the Schatten $\mathcal{S}_p(\mathcal{F}_\varphi^2)$ class for all p.
          \end{enumerate}
 \end{enumerate}
 \end{theorem}
 Observe that the validity of the result in  part (a) of (i) is dependent on the fast growth of the weight function $\varphi$ while part (b) does not. Part (a) of (ii) shows that cancellation property of the inducing maps plays no roll for compactness of the difference. On the contrary, it  is worth mentioning that compactness of the differences of two composition operators on the weighted Bergman spaces over the unit disc
  has been characterized by some suitable cancellation property of the inducing symbols  at each boundary points \cite{JMO}. Such property makes it possible for  each composition operator in the difference  not necessarily to be compact.

A natural question  following Theorem~\ref{thm4} is whether every isolated composition operator in  $C(\mathcal{F}_\varphi^p, \mathcal{F}_\varphi^p)$
is  still isolated under the essential norm topology which is weaker than the topology induced by the operator norm.  Our next  main result  shows  that  this is  in deed the case.
\begin{theorem} \label{thm5}
 Let $1\leq p<\infty$. Then a composition  operator $C_\psi$ in  $C(\mathcal{F}_\varphi^p, \mathcal{F}_\varphi^p)$ is essentially isolated if and only if it is isolated.
   \end{theorem}
   The isolated and essentially isolated points of the space of the  operators on the classical Fock spaces have not been also  identified as far as we know. As will be seen in Subsection~\ref{essential}, the method  we use to  prove  Theorem~\ref{thm5} can be easily adopted to the classical setting. In stead of using  the sequence of the  functions $f_{w, R}^*$, one can use  the sequence of the  normalized reproducing kernels  to  conclude the analogous results.
 \section{Preliminaries and auxiliary results}\label{background}
 Here we collect  background materials and present some auxiliary results which will be used to prove our main results in the sequel. Our first lemma shows that every symbol $\psi$ that induces a bounded operator on generalized Fock spaces  fixes a point in the  complex plane.
\begin{lemma}\label{lemfix}
Let $0<p, q<\infty $ and $\psi= az+b$ induces a bounded composition operator  $C_\psi:\mathcal{F}_\varphi^p \to \mathcal{F}_\varphi^q$. Then $C_\psi$ has a fixed point.
\end{lemma}
 \begin{proof}
By Theorem~\ref{thm1}, boundedness implies that $|a|\leq 1$ and $b= 0$ whenever  $|a|= 1$. Thus,  in the case when $a= 1$, $\psi$ fixes the origin. On the other hand, if $a\neq1$, then $\psi$ fixes the point $b/(1-a)$.
\end{proof}
In  the next section we will see that the  fixed point behaviour of the map $\psi$ plays an important role in proving our supercyclicity result.

 Another important ingredient in our subsequent consideration is the following.  By Proposition~A and Corollary~8 of \cite{Olivia} where the original idea comes from \cite{Borch},  for a sufficiently large positive number $R$, there exists a number $\eta(R)$ such that for any  $w\in \CC$ with $|w|> \eta(R)$, there exists an entire function $f_{(w, R)}$ such that
     \begin{align}
  \vspace{-0.3in}
  \label{test00}
      |f_{(w,R)}(z)| e^{-\varphi(z)}\leq C \min\Bigg\{ 1,\bigg(\frac{\min\{\tau(w), \tau(z)\}}{|z-w|}\bigg)^{\frac{R^2}{2}}\Bigg\}  \ \ \ \ \ \ \end{align}for all  $ z$ in $ \CC$,  and  for some constant $C$ that depends on $\psi$ and $R$. In particular,  when $z$ belongs to $D(w, R\tau(w))$, the estimate becomes
   \begin{align}
  \label{test0}
   |f_{(w,R)}(z)| e^{-\varphi(z)}\simeq 1,
    \end{align} where $D(a,r)$ denotes  the Euclidean disk centered at $a$ and radius $r>0$. Furthermore,  the functions
   $f_{(w, R)}$ belong to $\mathcal{F}_\varphi^p$  for all $p$ with  norms  estimated by
\begin{align}
\label{test}
\| f_{(w,R)}\|_{p}^p \simeq \tau(w)^2,\ \ \ \  \eta(R) \leq |w|.
\end{align}
For $p=2$, the space $\mathcal{F}_\varphi^2$ is known to be a reproducing kernel Hilbert space.
   An explicit expression for the kernel function is  still an interesting open problem. However, an   asymptotic estimation of the norm
   \begin{align}
   \label{asymptotic}
   \| K_w\|_{2}^2 \simeq  \tau(w)^{-2} e^{2\varphi(w)}.
   \end{align}holds for all $w\in \CC$. \\
      Furthermore, for subharmonic functions $\varphi$ and $f$, it also holds  a  local pointwise estimate
\begin{align}
\label{pointwise}
|f(z)|^p e^{-\beta \varphi(z)} \lesssim \frac{1}{\sigma^2\tau(z)^2} \int_{D(z, \sigma \tau(z))} |f(w)|^pe^{-\beta \varphi(w)} dA(w)
\end{align} for all finite exponent $p$,  any real number $\beta$,  and  a small positive number $\sigma$: see Lemma~7 of \cite{Olivia} for more details.
\subsection{Density of complex polynomials in $ \mathcal{F}_{\varphi}^{p}$} The cyclicity dynamical property of operators are closely related
to the density of polynomials in  various functional spaces; see for example \cite{PSB,GK,SH}.  This fact will be an   important tool  in proving   part (i) of Theorem~\ref{thm30} in our
current setting too.  Thus, we shall first   present  the following density  result.
\begin{lemma} \label{density}
 Suppose $ 0 < p < \infty $ and  $ f \in \mathcal{F}_{\varphi}^{p}$. Then there is a sequence of polynomials $\lbrace P_{n}\rbrace $ such that $ \Vert P_{n} - f \Vert_p \rightarrow 0  $ as $ n \rightarrow \infty .$
\end{lemma}
This lemma was first proved in \cite[Theorem 28]{Olivia}.  We provide here  another proof  using  the notions of inclusion and  dilation  techniques.
\begin{proof}
For  $ f \in \mathcal{F}_{\varphi}^{p}$ and  $ 0 < r < 1, $ define  a  sequence of  dilation functions $ f_{r} $ by $ f_{r}(z)= f (rz) $. Then  it  suffices to show that $ \|f_{r} - f \|_{p} \rightarrow 0  $ as $ r \rightarrow 1^{-} $ and  $\| f_{r} - P_{n} \|_{p} \rightarrow 0$ as $n\to \infty $ where $P_n$ is a sequence of complex polynomials. To show the first, we  may compute
\begin{align*}
  \| f_{r}\|^{p}_{p} =  \int_{\mathbb{C}} | f ( r z ) |^{p} e^{-p \varphi( z )}dA(z)     =  \frac{1}{r^{2}}\int_{\mathbb{C}} | f ( w ) |^{p} e^{-p\varphi( r^{-1}w )} dA(w) \\
  = \frac{1}{r^{2}} \int_{\mathbb{C}} | f ( w ) |^{p} e^{-p\varphi( w )}e^{-p ( \varphi ( r^{-1}w) + \varphi ( w ))} dA(w).
\end{align*}
Since $ \varphi $ is an  increasing radial function and $ 0 < r < 1 $ we have $ e^{-p ( \varphi ( r^{-1}w) - \varphi ( w ))} \leq 1 $ for all $w\in \CC$. Applying  Lebesque dominated convergence theorem,
\begin{align*}
 \lim_{r \rightarrow 1^{-}}\| f_{r}\|^{p}_{p}= \lim_{r \rightarrow 1^{-}} \frac{1}{r^{2}}\int_{\mathbb{C}}  | f ( w ) |^{p} e^{-p\varphi( w )}\Big(e^{-p ( \varphi ( r^{-1}w) - \varphi ( w ))} \Big) dA(w)= \| f \|_{p}^{p},
\end{align*}
showing that   $ \| f_{r}\|^{p}_{p} \rightarrow \| f \|_{p}^{p} $ and  hence $ f_{r}( z ) \rightarrow f( z ) $ as $ r \rightarrow 1^{-}$.\\
Therefore,
\begin{align}
\label{one}
\lim_{r \rightarrow 1^{-} }\|f_{r} - f \|_{p} =0.
\end{align}
Next, we fix  some $ r \in ( 0 , 1) $,   $ \alpha  \in ( r^2,\frac{1}{2})$ and proceed to show that
 $ f_{r} \in \mathcal{F}_{(\varphi, \alpha) }^{2}$ and  $ \mathcal{F}_{(\varphi , \alpha)} ^{2} \subset \mathcal{F}_{\varphi}^{p}$ where
\begin{align*}\mathcal{F}_{(\varphi, \alpha) }^{2}:=\left\{ f \; \text{entire} : \; \|f\|_{(2,\alpha)}^2= \int_{\mathbb{C}} | f ( z )|^{2} e^{-2\alpha \varphi ( z )} dA ( z ) < \infty \right\}.   \end{align*}
To  prove  the first, we may apply \eqref{pointwise} and  estimate
\begin{align*}
\| f_{r} \|_{(2, \alpha)}^{2}= \int_{\CC} |f(rw)|^2 e^{-2\alpha\varphi(w)} dA(w)\lesssim \| f \|_{p}^{2} \int_{\CC}  \frac{1}{\tau(rw)^{\frac{4}{p}}}e^{2\varphi(wr)-2\alpha\varphi(w)} dA(w).
\end{align*}
By definition of $\tau$ and $\varphi$, we also observe that
\begin{align*}\frac{1}{\tau(rw)^{\frac{4}{p}}}\lesssim e^{\varphi(wr)}\text{ as}\ \ |w| \to \infty.\end{align*} Taking this into account and the fact that $\alpha >r^2$ we further  estimate
\begin{align*}
\int_{\CC} \frac{1}{\tau(rw)^{\frac{4}{p}}}e^{2\varphi(wr)-2\alpha\varphi(w)} dA(w)\lesssim
\int_{\CC} e^{4\varphi(wr)-2\alpha\varphi(w)} dA(w)\\
\lesssim \int_{\CC} e^{2r^2\varphi(w)-2\alpha\varphi(w)} dA(w)  <\infty,
\end{align*} here  we used the fact that $\varphi$ grows faster than the classical function $|z|^2/2$ and hence  $\varphi(rw) \lesssim  \frac{r^2}{2} \varphi(w)$ whenever $|w| \to \infty$.

  For the inclusion property, we consider  $ h \in \mathcal{F}_{(\varphi, \alpha)}^{2}$ and  applying \eqref{pointwise} again  and proceed to estimate
\begin{align*}
 \int_{\CC}|h (z )|^{p} e^{-p\varphi( z )} dA (z )
    \lesssim  \| h \|_{(2,\alpha)}^{p} \int_{\CC} \frac{e^{p\alpha\varphi(z)-p\varphi(z)}}{\tau(z)^p} dA(z)\\
    \leq  \| h \|_{(2,\alpha)}^{p} \int_{\CC} e^{2p\alpha\varphi(z)-p\varphi(z)} dA(z)\lesssim \| h \|_{(2,\alpha)}^{p}.
\end{align*}
Now, since the  set of all holomorphic complex  polynomials  is  dense in the Hilbert space $ \mathcal{F}^{2}_{(\varphi, \alpha)}$, taking  $ P_{n}$ be the $n^{th}$  Taylor polynomial of $f_r$, we deduce from the inclusion property that
\begin{align*}
  \| f_{r} - P_{n} \|_{p} \leq C \| f_{r} -P_{n} \|_{(2, \alpha)}\rightarrow 0
\end{align*} as $n\to \infty$.  From this and  \eqref{one},  the result follows.
\end{proof}
The next lemma will find application in  proving  the equality of  the norm and essential norm of the composition operator when it acts on the generalized Hilbert space $\mathcal{F}_\varphi^2$.
\begin{lemma}\label{lem3}
 The normalized reproducing kernel $K_z/\|K_z\|_{2}$ converges weakly to $0$ in $\mathcal{F}_\varphi^2$ when $|z| \to \infty.$
 \end{lemma}
 \begin{proof}
 The sequence $e_n(z)= z^n/ \|z^n\|_{2}, \  n\geq  0$ represents the standard orthonormal basis  for $\mathcal{F}_\varphi^2$. It means that holomorphic polynomials are dense in $\mathcal{F}_\varphi^2$, and hence suffices to show that for any
nonnegative integer $m$
\begin{align*}
\Big |\Big\langle w^m, \frac{K_z}{\|K_z\|_{2}}\Big\rangle\Big|=\frac{|z|^m}{\|K_z\|_{2}} \to 0,   \ \ |z| \to \infty.
\end{align*}
  But this holds trivially as
\begin{align*}
\|K_z\|_{2}^2= \sum_{n=0} |e_n(z)|^2= \sum_{n=0}^\infty \frac{|z|^{2n}} {\|z^n\|_{2}^2},
\end{align*} which is a power series on $|z|^2$ with positive coefficients.
 \end{proof}
 We close this section with a lemma that will be used to prove Theorem~\ref{thm30} and Theorem~\ref{thm4} in the next section
 \begin{lemma} \label{lem4}
 Let $0< p, q <\infty, \ \ \psi_n(z)=a_nz+b_n $,  and $\psi(z)=az+b $ where  $(a_n) $ and $(b_n) $ are  sequences of complex  numbers   such that  $0\leq |a_n| \leq 1$ for all $n$,   $a_n \to a$ and $b_n \to b $  as $n \to \infty.$  Then for any  $f\in \mathcal{F}_\varphi^p$ and $ C_{\psi}, C_{\psi_n}\in C(\mathcal{F}_\varphi^p, \mathcal{F}_\varphi^q)$
\begin{align}
\label{uniform2}
 \lim_{n\to \infty}\|C_{\psi_n}f - C_{\psi}f\|_q =0.
\end{align}
     \end{lemma}
 \begin{proof} If $a_n= 0$ for all $n,$ then the lemma trivially follows. Thus,  assuming $0<|a|\leq 1$, we compute
 \begin{align*}
 \|C_{\psi_n}f \|_q^q= \int_{\CC}  |f(a_nz+b_n)|^q e^{-q\varphi(z)} dA(z)\quad \quad \quad \quad \quad \quad \quad \quad \quad\\
 = \int_{\CC}  |f(a_nz+b)|^q e^{-q\varphi(a_nz+b_n)} \Big( e^{q\varphi(a_nz+b_n)-q\varphi(z)} \Big) dA(z)\\
 = \int_{\CC}  |f(w) |^q e^{-q\varphi(w)}  \Big( |a_n|^{-2}e^{q\varphi(w)-q\varphi((w-b_n)/a_n)} \Big) dA(w).
\end{align*}
 Since $|a_n| \leq 1$, the quantity $ e^{q\varphi(w)-q\varphi((w-b_n)/a_n)} $ is uniformly bounded on $\CC$.  Applying Lebsegue dominated convergence theorem and smoothness of the weight function $\varphi$, we obtain
 \begin{align*}
 \lim_{n\to \infty}\|C_{\psi_n}f \|_q^q= \lim_{n\to \infty} \int_{\CC}  |f(w) |^q e^{-q\varphi(w)}  \Big( |a_n|^{-2}e^{q\varphi(w)-q\varphi((w-b_n)/a_n)} \Big) dA(w)\\
 = \int_{\CC}  |f(w) |^q e^{-q\varphi(w)}  \Big( |a|^{-2}e^{q\varphi(w)-q\varphi((w-b)/a)} \Big) dA(w)\\
 = \int_{\CC}  |f(az+b) |^q e^{-q\varphi(z)}    dA(z)= \|C_{\psi}f \|_q^q
\end{align*} from which \eqref{uniform2} follows.
 \end{proof}
  \section{Proof of the main results}\label{proofs}
 We now turn to the proofs of the main results.
   \subsection{Proof of Theorem~\ref{thm1}}
    We may first assume that  $0<p,q<\infty$ and  reformulate the boundedness and compactness problems of $C_\psi$ in terms of
     embedding maps between
  $\mathcal{F}_\varphi^p$ and $\mathcal{F}_\varphi^q$.  We   set a  pullback measure $\mu_{(\psi,q)}$  on $\CC$ as
  \begin{align}
  \label{pull}
  \mu_{(\psi,q)}(E)= \int_{\psi^{-1}(E)} e^{-q\varphi(w)} dA(w)
  \end{align} for every Borel subset $E$ of $\CC$.  Then we   observe
  \begin{align}
  \label{Carleson}
\|C_\psi f\|_{q}^q= \int_{\CC} |f(\psi(z))|^q e^{-q\varphi(z)}dA(z)=  \int_{\CC} |f(z)|^qd \mu_{(\psi,q)}(z).
 \end{align}
 From this,  it follows that $C_\psi:  \mathcal{F}_\varphi^p \to \mathcal{F}_\varphi^q$ is bounded if and only if the embedding map
 $i_d:  \mathcal{F}_\varphi^p \to L^q( \mu_{(\psi,q)})$ is bounded. To study  this reformulation  further, we may consider first part (i) of the theorem  along the   following two cases:\\
 \emph{ Case 1:} Assume  $0<p<q<\infty$. By  Theorem~1 of \cite{Olivia},  the map $i_d:  \mathcal{F}_\varphi^p \to L^q( \mu_{(\psi,q)})$ is bounded if and only if  for some $\delta >0,$
 \begin{align*}
 \sup_{w\in \CC} \frac{1}{\tau(w)^{2q/p}} \int_{D(w,\delta\tau(w))} e^{q\varphi(z)} d \mu_{(\psi,q)}(z)<\infty.
 \end{align*}
 Using \eqref{pull},  we may rewrite  this condition  again  as
 \begin{align}
 \label{bounded}
I:=\sup_{w\in \CC} \frac{1}{\tau(w)^{2q/p}} \int_{D(w,\delta\tau(w))} e^{q\varphi(z)} d \mu_{(\psi,q)}(z)\quad \quad \quad  \quad \quad \quad \quad \quad \quad  \nonumber\\
 = \sup_{w\in \CC} \frac{1}{\tau(w)^{2q/p}} \int_{D(w,\delta\tau(w))}e^{q( \varphi(z)-\varphi(\psi^{-1}(z))} dA(\psi^{-1}(z))<\infty.
\end{align}
Having singled out this  equivalent  reformulation, the next task is to examine  condition \eqref{bounded} and arrive at the assertion of the theorem. Let us  first  assume that \eqref{bounded} holds and show that $\psi(z)=az+b$ for some $|a|<1$. Applying  \eqref{pointwise} and estimating further on the right-hand side of \eqref{bounded} gives
\begin{align}
\label{finite}
I\gtrsim   \tau(\psi(w))^{\frac{2p-2q}{p}}  e^{q\big( \varphi((\psi(w)))-\varphi(w)\big)}
\end{align} for all $w$ in $\CC$ which  implies
\begin{align}
\label{one}
\tau(\psi(w))^{2\frac{(q-p)}{p}}\gtrsim e^{q\big( \varphi((\psi(w)))-\varphi(w)\big)}.
\end{align}
We claim that \begin{align*}
\limsup_{|w|\to \infty} (\varphi(\psi(w))-\varphi(w)) <0.
\end{align*}
If not, then  there exists a sequence $w_j\in \CC$ such that $|w_j| \to \infty$  as $j \to \infty $ and
\begin{align*}
\limsup_{j\to \infty} \varphi(\psi(w_j))-\varphi(w_j))\geq 0.
\end{align*}
This along with  \eqref{one} and applying the admissibility  assumptions on \eqref{assumption}, and the fact that $\psi$ is a  nonconstant entire function,   we get
\begin{align*}
0= \limsup_{j \to \infty}\tau(\psi(w_j))^{2\frac{(q-p)}{p}}\gtrsim \limsup_{j\to \infty}e^{q\big( \varphi((\psi(w_j)))-\varphi(w_j)\big)}.\nonumber\\
= e^{\limsup_{j\to \infty}q\big( \varphi((\psi(w_j)))-\varphi(w_j))\big)}\geq 1,
\end{align*}
 which is a contradiction. By the growth assumption on $\psi$  and \eqref{one} we  see that $\psi(z)= az+b$ for some $a$, $b$ in $\CC $ and $|a|<1.$

 Next, we assume that $\psi$ has the above linear form with  $|a|<1,$ and  proceed to show   that $C_\psi$  is a compact map. Using the  preceding embedding formulation and Theorem~1 of \cite{Olivia}, $C_\psi: \mathcal{F}_\varphi^p \to \mathcal{F}_\varphi^q $ is compact if and only if
 \begin{align}
 \label{compc}
 \lim_{|w|\to \infty}\frac{1}{\tau(w)^{2q/p}} \int_{D(w,\delta\tau(w))}e^{q\varphi(z)-q\varphi(\psi^{-1}(z))} dA(\psi^{-1}(z))= 0.
 \end{align}
  Since $|a|<1$, the integrand above is  a decaying function. Thus,
  \begin{align}
 \label{boundedd}
 \frac{1}{\tau(w)^{2q/p}} \int_{D(w,\delta\tau(w))}e^{q\varphi(z)-q\varphi(\psi^{-1}(z))} dA(\psi^{-1}(z))\nonumber\\
\lesssim  \frac{\tau(w)^2}{\tau(w)^{2q/p}} e^{q\varphi(w)-q\varphi(\psi^{-1}(w))}=  \tau(az+b)^{\frac{2p-2q}{p}} e^{q\varphi(az+b)-q\varphi(z))}.
\end{align}
By definition of $\varphi$ and $\tau$, we notice that the last  quantity in \eqref{boundedd} tends to zero as $|w|\to \infty$ and hence \eqref{compc} holds.
Since compactness obviously implies boundedness, we are finished with the proof for the case $p<q$.

\emph{Case 2}: $0<q<p<\infty$. Invoking the reformulation in \eqref{Carleson} again, $C_\psi:  \mathcal{F}_\varphi^p \to \mathcal{F}_\varphi^q$ is bounded (compact) if and only if the embedding map
 $i_d:  \mathcal{F}_\varphi^p \to L^q( \mu_{(\psi,q)})$ is bounded (compact). By Theorem~1 of \cite{Olivia},  boundedness or compactness of $i_d$  holds if and only   if  for some $\delta >0,$ the function
 \begin{align*}
 \mathcal{T}(z):= \frac{1}{\tau(z)^2} \int_{D(z,\delta\tau(z))} e^{q\varphi(w)} d \mu_{(\psi,q)}(w)= \frac{1}{\tau(z)^2}\int_{D(z,\delta\tau(z))}\frac{e^{q \varphi(w)}}{e^{q\varphi(\psi^{-1}(w))}} dA(\psi^{-1}(w))
 \end{align*} belongs to  $L^{\frac{p}{p-q}}(\CC, dA)$. We plan to show that this  holds if and only if $\psi$ has the form $\psi(z)= az+b$ with $|a|<1.$ Assuming the latter and applying H\"older's inequality
\begin{align*}
\int_{\CC}|\mathcal{T}(z)|^{\frac{p}{p-q}} dA(z)= \int_{\CC}\bigg(\frac{1}{\tau(z)^2}\int_{D(z,\delta\tau(z))}\frac{e^{q \varphi(w)}}{e^{q\varphi(\psi^{-1}(w))}} dA(\psi^{-1}(w))\bigg)^{\frac{p}{p-q}} dA(z)\nonumber\\
\lesssim \int_{\CC} \tau(z)^{-2} \int_{D(z,\delta\tau(z))} \frac{e^{\frac{qp}{p-q} \varphi(w)}}{e^{\frac{qp}{p-q}\varphi(\psi^{-1}(w))}} dA(\psi^{-1}(w))  dA(z)=:\mathcal{T}_1
\end{align*}
Since  $w\in  D(z,\delta\tau(z))$, by   Lemma~5 of \cite{Olivia}  there exists a positive  constant $c$  with \begin{align*}\frac{1}{c}\tau(w)\leq \tau(z)\leq c \tau(w).\end{align*}
 Then, for any $ \zeta \in D(z,\delta\tau(z))$
 \begin{align*}|\zeta-w| \leq |\zeta-z|+|z-w| \leq 2\delta\tau(z) \leq 2\delta c \tau(w)= \beta \tau(w), \ \ \beta:=2\delta c.\end{align*}
This shows that  $ D(z,\delta\tau(z)) \subset D(w,\beta\tau(w)) $ which together with Fubini's theorem and Lemma~5 of \cite{Olivia} again   imply
\begin{align*}
\mathcal{T}_1= \int_{\CC} \tau(z)^{-2} \int _{\CC}\chi_{D(z,\delta\tau(z))}(w) \frac{e^{\frac{qp}{p-q}\varphi(w)}}{e^{\frac{qp}{p-q}\varphi(\psi^{-1}(w))}} dA(\psi^{-1}(w))  dA(z)\quad \quad \quad \\
\leq \int_{\CC}\frac{e^{\frac{qp}{p-q}\varphi(w)}}{e^{\frac{qp}{p-q}\varphi(\psi^{-1}(w))}} \left( \int _{\CC}\chi_{D(w,\beta\tau(w))}(z) \tau(z)^{-2}dA(z)\right) dA(\psi^{-1}(w))   \\
  =  \int_{\CC}\frac{e^{\frac{qp}{p-q}\varphi(w)}}{e^{\frac{qp}{p-q}\varphi(\psi^{-1}(w))}} \left( \int _{D(w,\beta\tau(w))} \tau(z)^{-2}dA(z)\right) dA(\psi^{-1}(w))\\
\simeq  \int_{\CC}\frac{e^{\frac{qp}{p-q}\varphi(w)}}{e^{\frac{qp}{p-q}\varphi(\psi^{-1}(w))}}  dA(\psi^{-1}(w)) <\infty.
\end{align*}
On  the other hand, if $ \mathcal{T}$ is  $L^{\frac{p}{p-q}}$ integrable over $\CC$, then $C_\psi: \mathcal{F}_\varphi^p \to \mathcal{F}_\varphi^q  $ is bounded, and applying $C_\psi$ to the sequence of test functions
$f_{(w, R)}$ and using a weaker version  of the point estimate in  \eqref{pointwise}
\begin{align*}
\|f_{(w, R)}\|_p \gtrsim \|C_\psi f_{(w, R)} \|_q \gtrsim  | f_{(w, R)}(\psi(z))|\tau(z)^{\frac{2}{q}} e^{-\varphi(z)}
\end{align*} for all points $w, z \in \CC$. Setting, in particular, $w= \psi(z)$  and  invoking the estimates in \eqref{test0} and  \eqref{test} gives
\begin{align}
\label{change}
\tau(\psi(z))^{\frac{2}{p}} \gtrsim \tau(z)^{\frac{2}{q}} e^{\varphi(\psi(z))-\varphi(z)}.
\end{align}  Since $\psi$ is a nonconstant entire function, the left-hand side of \eqref{change} tends to zero as $|z|\to \infty.$ So does the right-hand side and that happens only if
\begin{align}
\label{change1}
\sup_{z\in \CC} e^{\varphi(\psi(z))-\varphi(z)} <\infty.
\end{align}
Arguing as in the proof of the corresponding part in case 1, we observe that \eqref{change1} holds only if  $\psi$  has  a  linear form
$\psi(z)= az+b $  with $|a|\leq1 $ and $b= 0$ whenever $|a|=1$. We further claim that $|a|<1$. If not,  using again the $L^{\frac{p}{p-q}}$ integrability of $ \mathcal{T}$
\begin{align*}
\int_{\CC}|\mathcal{T}(z)|^{\frac{p}{p-q}} dA(z)= \int_{\CC} \tau(z)^{-\frac{2p}{p-q}} \bigg(\int_{D(z,\delta\tau(z))}\frac{e^{q \varphi(w)}}{e^{q\varphi(\psi^{-1}(w))}} dA(\psi^{-1}(w))\bigg)^{\frac{p}{p-q}}dA(z)\nonumber\\
\gtrsim \int_{\CC} \tau(z)^{-\frac{2p}{p-q}} \bigg(\int_{D(z/a,\delta\tau(z))}\frac{e^{q \varphi(aw)}}{e^{q\varphi(w)} }dA(w)\bigg)^{\frac{p}{p-q}}dA(z)\nonumber\\
= \int_{\CC} \tau(z)^{-\frac{2p}{p-q}} \tau(z/a)^{\frac{2p}{p-q}} dA(z)=\infty
\end{align*} which is  a contradiction as $ \tau(z/a)= \tau(z)$ whenever $|a|=1$ .

(ii) The corresponding proofs for  this part  follows rather easily by simply setting $p=q$ in the arguments made in part (i). Thus, we  skip it.
\subsection{Proof of Theorem~\ref{thm2}}
 (i) If $C_\psi $ is bounded but not compact, then by Theorem~\ref{thm1},  $\psi(z)= az$ where $|a|=1$. Consequently, $\varphi(\psi(z))= \varphi(az)= \varphi(|az|)= \varphi (z)$. With this, we  find an upper bound for  the norm of the operator
 \begin{align*}
  \|C_\psi f\|_{p}^p=\int_{\CC} \frac{|f(\psi(z))|^p}{e^{p\varphi(z)}} dA(z) \leq  \sup_{z\in \CC} \Big( e^{p\varphi(\psi(z))-p\varphi(z)}
\Big)  \int_{\CC} \frac{| f(\psi(z))|^p}{e^{p\varphi(\psi(z))}} dA(z)\\
= \sup_{z\in \CC}  e^{p\varphi(\psi(z))-p\varphi(z)}
  \| f\|_{p}^p=
    \| f\|_{p}^p. \ \ \quad
 \end{align*}
Therefore,
 \begin{align}
 \label{uppper}
 1\geq \|C_\psi \|\geq
 \|C_\psi \|_{e}.
 \end{align}
A common way to prove lower bounds for  essential norms
is to find a suitable weakly  null  sequence of functions $f_n$ and use  the fact that
\vspace{-0.02in}
\begin{align}
\label{common}
\|C_\psi\|_e \geq \limsup_{n\to \infty} \|C_\psi f_n\|_{p}.
\end{align}
On classical Fock spaces, the sequence of the reproducing kernels does  this job. Since no explicit expression is known for the
  kernel function in our current setting, we will instead use  the sequence of functions  \begin{align}
  \label{unitt}f_{(w, R)}^*=f_{(w, R)}/\|f_{(w, R)}\|_{p} \end{align}   as described by the properties in
\eqref{test00}, \eqref{test0}, and \eqref{test}. Obviously, the sequence $ f_{(w, R)}^* $ is uniformly bounded,   and due to the relation in \eqref{test00},  $ f_{(w, R)}^* \rightarrow 0 $ uniformly on compact subset of $\mathbb{C}$ as $|w|\rightarrow \infty.$ Thus,  $ f_{(w, R)}^* \rightarrow 0 $ weakly as $|w|\rightarrow \infty .$    With this,  we proceed  to make further estimates  on the right-hand side of  the norm in \eqref{common}. Making use of \eqref{pointwise} for some small  positive number $\delta$
\begin{align*}
\|C_\psi\|_e \geq \limsup_{|w|\to \infty} \big\| C_\psi f^*_{((\psi(w),R)}\big\|_{p} \quad \quad \quad \quad \quad \quad \quad \quad \quad \quad \quad \quad \quad \quad \quad \quad \quad  \\
\simeq \limsup_{|w| \to \infty} \frac{1}{\tau(w)^{\frac{2}{p}}} \Bigg(\int_{\CC}
|f_{((\psi(w),R)}(\psi(z))|^pe^{-p\varphi(z)}dA(z)\Bigg)^{\frac{1}{p}}\nonumber\\
 \geq \limsup_{|w|\to \infty} \frac{1}{\tau(w)^{\frac{2}{p}}} \Bigg( \int_{D(\psi(w),\delta \tau(\psi(w)))} |f_{((\psi(w),R)}(\psi(z))|^pe^{-p\varphi(\psi(z))}dA(z)\Bigg)^{\frac{1}{p}}\nonumber\\
 \gtrsim \limsup_{|w|\to \infty} \frac{\tau(\psi(w))^{\frac{2}{p}}|f_{((\psi(w),R)}(\psi(w))|^pe^{-p\varphi(\psi(w))}}{\tau(w)^{\frac{2}{p}}}\\
 \simeq \limsup_{|w|\to \infty} \frac{\tau(\psi(w))^{\frac{2}{p}}}{\tau(w)^{\frac{2}{p}}}= \limsup_{|w|\to \infty} \frac{\tau(w)^{\frac{2}{p}}}{\tau(w)^{\frac{2}{p}}}=1
 \end{align*} which  completes the proof of the lower estimate.

For the Hilbert space case, applying Lemma~\ref{lem3}, we have
\begin{align*}
\|C_\psi\|_e \geq \limsup_{|w|\to \infty} \big\|\|K_w\|_{2}^{-1} C_\psi K_w\big\|_{2} \quad \quad \quad \quad \quad \quad \quad \quad \quad \quad \quad \quad \quad \quad \quad \quad \quad \nonumber  \\
=  \limsup_{|w|\to \infty} \|K_w\|_{2}^{-1}\bigg( \int_{\CC} |K_w(\psi(z))|^2 e^{-2\varphi(z)} dA(z)\bigg)^{1/2}\nonumber\\
=  \limsup_{|w|\to \infty} \|K_w\|_{2}^{-1}\bigg( \int_{\CC} |K_w(az)|^2 e^{-2\varphi(az)} dA(z)\bigg)^{1/2}= 1,
\end{align*} from which and \eqref{uppper} we arrive at the  asserted equality.

 (ii) Since  Schatten class membership has the nested property in the sense that    $\mathcal{S}_p\subseteq  \mathcal{S}_q$ for  $p\leq q$, it suffices to verify  the theorem  only for the case  when  $p$ is in the range $0<p<2$. Recall that a compact operator $T$ belongs to the Schatten $\mathcal{S}_p$ class if and only if  the positive operator  $(T^*T)^{p/2}$ belongs to the trace class $\mathcal{S}_1$.  Furthermore, $T \in \mathcal{S}_p$ if and only if $T^* \in S_p,$ and $\|T\|_{\mathcal{S}_p} = \|T^*\|_{\mathcal{S}_p} $. Thus, we may  estimate the trace of $\big( C_\psi C_\psi^*\big)^{p/2}$ by
\begin{align}
\label{help}
\text{tr}\big((C_\psi C_\psi^*)^{\frac{p}{2}} \big)= \int_{\CC} \Big\langle \big(C_\psi C_\psi^*  k_{z}\big)^{\frac{p}{2}}, k_{z}\Big \rangle dA(z)
\leq \int_{\CC}\Big\langle C_\psi C_\psi^* k_{z}, k_{z}\Big\rangle^{\frac{p}{2}}dA(z)\nonumber\\
=\int_{\CC} \| C_\psi^*k_{z}\|_{2}^p dA(z),
\end{align} where the inequality  holds  since $0<p\leq 2, \ \ C_\psi C_\psi^* $ is a positive operator,  and $k_{z}= K_z/\|K_z\|_2$ is a unit norm vector, see \cite[Proposition 1.31]{kz}. On the other hand, by the reproducing property of the kernel function, we have the adjoint property
\begin{align*}
C_\psi^*K_{w}(z)= \big\langle C_\psi^*K_{w}, K_{z}\big\rangle= \big\langle K_{w}, C_\psi K_{z}\big\rangle
= \overline{\big\langle C_\psi K_{z},K_{w} \big\rangle}=  K_{\psi(w)}(z).
\end{align*}
From this  estimate and \eqref{asymptotic}, we  have that
\begin{equation*}
\|C_\psi^* k_{w}\|_{2}\simeq  \frac{\tau(w)}{\tau(\psi(w))} e^{\varphi(\psi(w))-\varphi(w)}.
\end{equation*}
This along with \eqref{help} and compactness of $C_\psi$  implies
\begin{align*}
\text{tr}\big(\big(C_\psi C_\psi^*\big)^{\frac{p}{2}} \big)\leq  \int_{\CC }\bigg(\frac{\tau(w)}{\tau(\psi(w))}\bigg)^{p} e^{p(\varphi(\psi(w))-\varphi(w))}dA(z) \quad \quad \quad \quad \quad \quad \nonumber\\
= \int_{\CC }\bigg(\frac{\tau(w)}{\tau(aw+b)}\bigg)^{p} e^{p(\varphi(\psi(w))-\varphi(w))}dA(z)
\lesssim  \int_{\CC } e^{p(\varphi(\psi(w))-\varphi(w))}dA(z)<\infty,
\end{align*} from which and
 condition \eqref{help},  we conclude that  $\text{tr}\big(\big(C_\psi C_\psi^*\big)^{\frac{p}{2}} \big)$ is finite.
\subsection{Proof of Theorem~\ref{thm3}}
Obviously (iv) implies (i). On the other hand, a unitary operator has inverse equal to its adjoint, and also  an invertible operator commute with its inverse which gives  (iii) implies (iv). Thus, we shall proceed to  show that (i) implies (ii) and (ii) implies (iii). To this end, if  $C_\psi$ is hyponormal, then applying \eqref{exact} and  the adjoint property again
\begin{align*}
\|K_z\|_{2}^2\geq  \|C_\psi K_z\|_{2}^2\geq \|C_{\psi}^* K_z\|_{2}^2 = \|K_{\psi(z)}\|_{2}^2.
\end{align*}
From this and the asymptotic relation in \eqref{asymptotic}, we further have
\begin{align*}
\frac{e^{2\varphi(z)}}{\tau(z)^2} \gtrsim \frac{e^{2\varphi(az+b)}}{\tau(az+b))^2}
\end{align*}  and hence
\begin{align}
\label{normal}
\tau(az+b)^2 \gtrsim \tau(z)^2 e^{2\varphi(az+b)-2\varphi(z)}.
\end{align}
By definition of $\tau$ and the admissibility condition on the weight function $\varphi$, the  inequality  in  \eqref{normal} holds when $|z|\to \infty$ only if $b= 0$. Then boundedness of the operator implies that $|a|= 1$.

On the other hand,  if condition (ii) holds, then   $b= 0$ and  $C_\psi(z)= az, $ with $ |a|=1$.  We need to show that  $C_\psi$ is surjective and preserves the inner product on $\mathcal{F}_\varphi^2$. Thus, for each $f, g$ in  $\mathcal{F}_\varphi^2$:
\begin{align*}
\langle C_\psi f, C_\psi g \rangle=  \int_{\CC}f(az) \overline{g(az)} e^{-2\varphi(z)} dA(z)\quad \quad \quad \quad  \quad \quad \quad \quad\nonumber\\
=\frac{1}{|a|^2} \int_{\CC}f(w) \overline{g(w)} e^{-2\varphi(w)} dA(w)= \langle  f,  g \rangle.
\end{align*} which shows that the operator preserves the inner product. It remains to show that the operator is also surjective. But this follows easily since $C_{\psi}^{-1}= C_{\psi^{-1}}$ exists in this case.
\subsection{Proof of Theorem~\ref{thm30}}
\emph{Part(i)}. As pointed  earlier, this part of the theorem  was proved for the special case $p= 2$ and $\varphi(z)= z^s, s\leq 1$ in [Theorem 4.2] \cite{GK}. The proof in \cite{GK} is   based on
Hilbert space properties. In the proof to follow, we will follow the same approach but replaces all the  Hilbert space arguments by other general arguments.

 Let us first  assume that $C_\psi$ is cyclic  and prove  the necessity of the  condition. Arguing on the contrary, if $ a^{k} = a   $ for some $ k \geq 2  $, then  $|a|=1$ and hence  $\psi(z)= az$.  For any cyclic vector $f_0$ in $ \mathcal{F}_{\varphi}^p$,  it  follows that
 $C^{k}_{\psi}f_0( z ) = f_0 ( a^{k}z )= f_0 ( az ) = C_{\psi} f_0 ( z ) $
  which implies \begin{align*} \{ C_{\psi}^{n} f_0 , n \in {\mathbb{Z}_+}  \} =  \{ C_{\psi}^{n} f_0 : \; n =0,  1, 2, 3, . . . k  \}. \end{align*}
   This shows that  the closed linear span of the  orbit  is finite dimensional, and hence $ C_{\psi} $ can not be cyclic.

Conversely, suppose  $ \psi ( z ) = a z+b $  and $ a^{n} \neq a   $ for every $ n \geq 2 $ which obviously implies that $a\neq1$. Then  we proceed to show that there exists
    a cyclic vector  $ h \in \mathcal{F}_{\varphi}^p $  with  Taylor  series expansion at $z= \frac{b}{1-a}$
    \begin{align*}  h(z)=
  \sum_{n=0}a_n \bigg(z-\frac{b}{1-a}\bigg)^n.
      \end{align*}
   Let us first make a short argument verifying   the necessity that for $h$ to be a  cyclic  vector, $a_{n} \neq 0 $ for all $  n \in {\mathbb{Z}_+}$. If $a_n= 0$ for some $n= m$, it follows from the fact that
   \begin{align*}
   C_\psi^k h(z)= \sum_{n=0}^{\infty}a_{n} a^{kn}\Big(z-\frac{b}{1-a}\Big)^{n},
   \end{align*}  all functions $f$ in the closed linear span of $
   \ \left\{ C_{\psi}^{k} h : k \in {\mathbb{Z}_+}\right\} $  satisfy  $ \frac{d^m}{dz^m} f\big |_{z=\frac{b}{1-a}}= 0 $
    which   contradicts  the cyclic behaviour of $h$.

  We may now consider the case when $|a|=1$ and hence $b=0$.  This together with the assumption  $ a^{n} \neq a   $ for every $ n \geq 2 $  imply
  \begin{align*} \overline{\left\{ a^{k},  k\in {\mathbb{Z}_+}  \right\}} = \mathbb{T}= \{ z \in \mathbb{C}: | z |= 1\}.
   \end{align*}
 Thus, for each  $ w \in \mathbb{T} $ there exists  a sequence $ \{ k_{j} \}_{j} \; \text{in} \; \mathbb{Z_{+}}$ such that $ a^{k_{j}} \rightarrow w \; \text{as} \; j \rightarrow \infty.  $
Let $ \psi_{w} ( z ) = w z. $ Then we claim
\begin{align}
\label{uniform}
\lim_{j \rightarrow \infty } \| C_{\psi}^{k_{j}}h - C_{\psi_{w}} h \|_p=0.
\end{align}
Using the radial property $\varphi(a^{k_j}z)= \varphi(z)$ and change of variables,  we compute
\begin{align*}
\lim_{j \to \infty} \|C_{\psi}^{k_{j}}h\|_p^p=  \lim_{j\to \infty}\int_{\CC} |h(a^{k_j}z)|^p e^{-p\varphi(a^{k_{j}}z)}dA(z)\quad \quad \quad \quad \quad \quad \quad \\
= \lim_{j\to \infty}\frac{1}{|a^{k_{j}}|^2} \int_{\CC} |h(z)|^p e^{-p\varphi(z)} dA(z)
=\frac{1}{|w|^2} \int_{\CC} |h(z)|^p e^{-p\varphi(z)} dA(z)\\
= \int_{\CC} |h(wz)|^p e^{-p\varphi(z)} dA(z)= \|C_{\psi_w}h\|_p^p
\end{align*} from which \eqref{uniform} follows.  This  verifies that $ C_{\psi_{w}} h$ belongs to the closed linear span of  $  \left\{ C_{\psi}^{k} h : k \in {\mathbb{Z}_+}\right\}. $ \\
The mapping $ G : \mathbb{T} \rightarrow  \mathcal{F} _{\varphi}^{p} $ defined by $ G (w )= C_{_{\psi_{w}}} h $ is continuous, which can be extended to analytic function $ \tilde{G}$ in $ \mathbb{D}$ with  $ \tilde{G}(w) = G ( w )$ on the boundary of  $ \mathbb{D}$.
 Then,  by Cauchy Integral Formula ( using $ C_{\psi_{w}}( z ) = G ( w )(z) =\tilde{G}(w)(z) $)
 \begin{align*} a_{n} z^{n} = \dfrac{1}{2\pi i}\int_{|w|=1}\dfrac{C_{\psi_{w}}h ( z )}{w^{n+1}} dw.
  \end{align*}
 Hence the set of  polynomials  $ a_{n} z^{n}, \ \ n\in {\mathbb{Z}_+} $ belongs to  the closed linear span of  $  \left\{ C_{\psi}^{k} h : k \in {\mathbb{Z}_+}\right\}. $  From this,  the fact that  $a_n\neq 0 $ for all $n \in {\mathbb{Z}_+}$, and
Lemma~\ref{density}, the  conclusion of the theorem follows  for this case.

It remains to show the case when $0<|a|<1$. For each  $ m \in {\mathbb{Z}_+}$,  we decompose  the function $h$ as  $h= h_m+g_m$ where
\begin{align}\label{split}
h_{m}( z ) = \sum_{n=0}^{m}a_{n}\left(  z -\frac{b}{1-a}\right)^{n} \hspace{0.4cm} \text{and} \  \hspace{0.4cm} g_{m}( z ) = \sum_{n=m+1}^{\infty}a_{n} \left(  z -\frac{b}{1-a}\right)^{n}.
\end{align}
 Using induction we plan to prove that for every  $m \in {\mathbb{Z}_+}$  \begin{align*} h_{m} \in  \overline{\text{span} \left\{ C_{\varphi}^{k} h : k \in {\mathbb{Z}_+}\right\}}.  \end{align*}
To this end, consider a function  $ g $ in $ \mathcal{F}_{\varphi}^{p}$ and observe that
\begin{align*} C_{\psi}^{k}g ( z ) = g\left(a^{k} z + \dfrac{b(1 - a^{k})}{1-a}\right).  \end{align*} Since $ | a |< 1 ,$ we  also have
$ a^{k} z + (1-a)^{-1}b(1 - a^{k})\rightarrow (1-a)^{-1}b$  and by  Lemma~\ref{lem4}
\begin{align*}
  \lim_{k\to \infty }\| C_{\psi}^{k}g - C_{\frac{b}{1-a}}g \|_{p} = 0.
\end{align*}
It follows form  this and \eqref{split} that
 \begin{align*}
  \lim_{k \to  \infty}\| C_{\psi}^{k}g_{0} - C_{\frac{b}{1-a}}g_{0}  \|_{p} = \lim_{k \to  \infty} \| C_{\psi}^{k}g_{0} \|_{p} =0
 \end{align*} from which we  further deduce
  \begin{align*} \| C_{\psi}^{k}h - a_{0} \|_p = \| C_{\psi}^{k}( a_{0} + g_{0} ) - a_{0} \|_p \leq \| C_{\psi}^{k}( a_{0}  ) - a_{0} \|_p + \| C_{\psi}^{k}(  g_{0} )  \|_p \rightarrow 0 \end{align*}
   as $ k \rightarrow \infty.$
Therefore,
 \begin{align*} h_{0}\in \overline{\text{span} \left\{ C_{\psi}^{k} h : k \in {\mathbb{Z}_+}\right\}}.
  \end{align*}
Suppose now  that $ h_{0}, h_{1}, . . . h_{N-1}\in \overline{\text{span} \left\{ C_{\psi}^{k} h : k \in {\mathbb{Z}_+}\right\}}.  $
Then by  the decomposition in \eqref{split}  it holds that $  g_{N-1}\in \overline{\text{span} \left\{ C_{\psi}^{k} h : k\in {\mathbb{Z}_+}\right\}},$  and hence
\begin{align}\label{cycle}
  C_{\psi}^{j} g_{N-1}\in \overline{\text{span} \left\{ C_{\psi}^{k} h : k \in {\mathbb{Z}_+}\right\}}
\end{align}
for every $ j \in {\mathbb{Z}_+}$.  We next compute
  \begin{align}
   C_{\psi}^{j} g_{N-1}( z )= C_{\psi}^{j} \sum_{n=N}^{\infty} a_{n} \left(z - \dfrac{b}{1 - a}\right)^{n}
      = \sum_{n=N}^{\infty} a_{n} a^{jn} \left(z - \dfrac{b}{1 - a}\right)^{n} \quad \quad \nonumber\\
    = a^{jN}\left(z - \dfrac{b}{1 - a}\right)^{N} \sum_{n=N}^{\infty} a_{n} a^{j(n-N)} \left(z - \dfrac{b}{1 - a}\right)^{n-N}\quad \quad \nonumber\\
    = a^{jN}\left(z - \dfrac{b}{1 - a}\right)^{N} C_{\psi}^{j}\sum_{n=N}^{\infty} a_{n}  \left(z - \dfrac{b}{1 - a}\right)^{n-N}
         =a^{jN}\left(z - \dfrac{b}{1 - a}\right)^{N} C_{\psi}^{j} f_{N-1}(z)
         \label{cycle1}
 \end{align}
 where $\psi^j(z)= a^jz + \frac{b(1-a^j)}{1-a}$ and
 \begin{align*}
   f_{N-1}(z)  =  a_N+ \sum_{n=N+1}^{\infty} a_{n}  \left(z - \dfrac{b}{1 - a}\right)^{n-N}.
 \end{align*}
  From \eqref{cycle} and \eqref{cycle1} we  also  obtain
   \begin{align}
   \label{cyclic3}
     \left(z - \dfrac{b}{1 - a}\right)^{N} C_{\psi}^{j} f_{N-1} \in \overline{\text{span} \left\{ C_{\psi}^{k} h : k \in {\mathbb{Z}_+}\right\}}.
   \end{align}
   By Lemma~\ref{lem4}  we have that
   \begin{align}
   \label{null}
   \lim_{j\to \infty} \|C_{\psi^{j}} f_{N-1} - a_{N}\|_p= \lim_{j\to \infty} \|C_{\psi^{j}} f_{N-1} -  C_{\frac{b}{1-a}}f_{N-1} \|_p= 0.
   \end{align}
  To this end, we  further  claim that
  \begin{align}
  \label{cyclic4}
   \Gamma_j(z):= \left(z - \dfrac{b}{1 - a}\right)^{N} C_{\psi^{j}} f_{N-1} \rightarrow a_{N}\left(z - \dfrac{b}{1 - a}\right)^{N}=:\Gamma(z)
  \end{align}
in $ \mathcal{F}_{\varphi}^{p}$ as $ j \rightarrow \infty $ as well.   We may  compute
\begin{align*}
\|\Gamma_j\|_p^p= \int_{\mathbb{C}}\bigg| \left(z - \dfrac{b}{1 - a}\right)^{N} C_{\psi}^{j} f_{N-1}(z)\bigg|^{p}e^{-p\varphi(z)} dA(z)\quad \quad \quad \quad \quad \quad  \\
= \int_{\mathbb{C}}\bigg| f_{N-1}\Big(a^jz+\frac{b(1-a^j)}{1-a}\Big)\bigg|^{p}e^{-p\varphi\big(a^jz+\frac{b(1-a^j)}{1-a}\big)}U_j(z)dA(z)
\end{align*} where
\begin{align*}
U_j(z)=\bigg|z - \dfrac{b}{1 - a}\bigg|^{pN} e^{p\varphi\big(a^jz+\frac{b(1-a^j)}{1-a}\big)-p \varphi(z)}
\end{align*}
We also observe that since $\varphi$ is an increasing weight function, and $|a^j|<1$, the sequence of functions $U_j$ are uniformly bounded over $\CC$. Furthermore, since norm convergence in $ \mathcal{F}_{\varphi}^{p}$ implies pointwise convergence,  by \eqref{null} for each $z\in\CC$
\begin{align*}
      C_{\psi^{j}} f_{N-1}(z) \to    C_{\frac{b}{1-a}}f_{N-1}(z)
   \end{align*}  as $ j \rightarrow \infty $. With  this,  an application of Lebesques convergence theorem implies
 \begin{align*}
\lim_{j\to \infty}\|\Gamma_j\|_p^p= \lim_{j\to \infty} \int_{\mathbb{C}}\bigg| f_{N-1}\Big(a^jz+\frac{b(1-a^j)}{1-a}\Big)\bigg|^{p}e^{-p\varphi\big(a^jz+\frac{b(1-a^j)}{1-a}\big)}U_j(z)dA(z)\\
=  \int_{\mathbb{C}}\Big| C_{\frac{b}{1-a}}f_{N-1}(z)\Big|^{p}  \bigg|z - \dfrac{b}{1 - a}\bigg|^{pN} e^{-p\varphi(z)}dA(z)=\|\Gamma\|_p^p.
\end{align*}
    Thus, the claim in \eqref{cyclic4} follows which along with   \eqref{cyclic3} give
\begin{align*}
 a_{N}\left(z - \dfrac{b}{1 - a}\right)^{N} \in \overline{\text{span} \left\{ C_{\psi}^{k} h : k \in {\mathbb{Z}_+}\right\}},\ \text{and}\ \  h_{N} \in \overline{\text{span} \left\{ C_{\psi}^{k} h : k \in {\mathbb{Z}_+}\right\}},
\end{align*}
Therefore,
\begin{align*}
 h_{m} \in \overline{\text{span} \left\{ C_{\psi}^{k} h : k \in {\mathbb{Z}_+}\right\}},
\end{align*}
for every $ m \in {\mathbb{Z}_+} $ which  in turn results in
\begin{align*}
 a_{n}\left(z - \dfrac{b}{1 - a}\right)^{n} \in \overline{\text{span} \left\{ C_{\psi}^{k} h : k \in {\mathbb{Z}_+}\right\}}
\end{align*}
for every $ n \in {\mathbb{Z}_+}$. Then, since  $a_n\neq 0$ for all $n\in {\mathbb{Z}_+}$, by Lemma~\ref{density}  the assertion of the  theorem follows.

\emph{Part (ii)}. We now proceed to show that $C_\psi$ can not be supercyclic. We set  $\psi(z)= az+b $ and argue in the  direction of contradiction, and assume that $C_\psi$  has a supercyclic vector $f\in \mathcal{F}_\varphi^p$. If $0<|a|<1$, then by
Lemma~\ref{lemfix}, $\psi$ fixes the point $b/(1-a)$. It follows that  $f(b/(1-a))\neq 0$. If not, the projective orbit contains only functions which vanishes at $b/(1-a)$.  Now for  each function $g$ in the projective orbit of $f$, there exists a sequence $(\lambda_{n_{k}})$ such that
\begin{align*} \lim_{k \to \infty}\|
\lambda_{n_{k}} C_\psi^{{n_{k}}} f - g\|_p= 0.
\end{align*}
Then we compute
\begin{align*}
g\Big(\frac{b}{1-a}\Big)= \lim_{k\to \infty}\lambda_{n_{k}} C_\psi^{n_{k}} f\Big(\frac{b}{1-a}\Big)= \lim_{k\to \infty}\lambda_k C_{\psi^{n_{k}}} f\Big(\frac{b}{1-a}\Big) = f\Big(\frac{b}{1-a}\Big)\lim_{k\to \infty}\lambda_{n_{k}},
\end{align*} where we used here  the fact that norm convergence implies pointwise convergence. Thus, for all  $z \in \CC$, applying the fact that  $a^{n_{k}}\to 0$ as $k\to \infty$ and \eqref{dependent}
\begin{align*}
g(z)= \lim_{k\to \infty}\lambda_{n_{k}} C_\psi^{n_{k}} f(z)= \lim_{k\to \infty}\lambda_{n_{k}} f \Big( a^{n_{k}} z+ \frac{b(1-a^{n_{k}})}{1-a}\Big) \quad \quad \quad \quad \quad \quad \quad  \quad \quad \quad   \\
 =\left[f\left(\frac{b}{1-a}\right)\right]^{-1} g\Big(\frac{b}{1-a}\Big)\lim_{k\to \infty} f \Big( a^{n_{k}} z+ \frac{b(1-a)}{1-a}\Big)\quad \quad \quad \quad \quad\quad  \\
 =  \left[f\left(\frac{b}{1-a}\right)\right]^{-1} g\Big(\frac{b}{1-a}\Big)f\Big(\frac{b}{1-a}\Big) = g\Big(\frac{b}{1-a}\Big),
\end{align*}
 showing that  only constant functions are in the projective orbit of $f$  resulting a contradiction.

 If $\psi(z)= az$ with $|a|=1$, then it fixes the origin. We may choose a univalent function  $g \in \mathcal{F}_\varphi^p$ such that $g(0)\neq 0$, and  pick a  subsequence $\psi^{n_k}$ such that  $\psi^{n_k}(z) \to a z$ as $k\to \infty.$ Then
 \begin{align*}
 g(z)=  \lim_{k\to \infty}\lambda_{n_{k}} C_\psi^{n_{k}} f(z)= \lim_{k\to \infty}\lambda_{n_{k}} f(a^{n_k} z)=  g(0) f(az).
 \end{align*}
 It follows that
 \begin{align}
 f(z)= \frac{f(0)}{g(0)} g\Big(\frac{z}{a}\Big)
 \end{align} is univalent as $f(0)\neq 0$.  Consequently, the projective orbits of $f$ contains only univalent functions which is again  a contradiction.
\subsection{Proof of Theorem~\ref{thm4}} We consider first  (a) of  part (i) and assume that  $p\neq q$.  We  plan to show that $C(\mathcal{F}_\varphi^p, \mathcal{F}_\varphi^q )$ is  connected. Aiming  to argue in the direction of  contradiction, suppose there exists an isolated point $C_\psi \in C(\mathcal{F}_\varphi^p, \mathcal{F}_\varphi^q )$. Since $p\neq q$, by Theorem~\ref{thm1},  $C_\psi$ is a compact operator and hence $\psi(z)= az+b, \ |a|<1.$ Then, choose two sequences of numbers $(a_n) $  with $|a_n| <1$ and $a_n\neq 0 $ for all $n$ and $b_n$ such that $a_n \to a$ and $b_n \to b $  as $n \to \infty.$  It follows that $\psi_n(z)=a_nz+b_n \to \ az+b= \psi(z)$. Then for any  $f\in \mathcal{F}_\varphi^p$, by Lemma~\ref{lem4}
\begin{align*}
 \|C_{\psi_n}f - C_{\psi}f\|_q \to 0 \ \ \text{as} \ \ n\to \infty.
\end{align*}
 Using this we find
  \begin{align*}
 \lim_{n\to \infty} \|C_{\psi_n}- C_\psi\| \leq  \lim_{n\to \infty} \sup_{\|f\|_{p\leq 1 }}\|C_{\psi_n}f - C_{\psi}f\|_q \quad \quad \quad \quad\quad \quad \quad \quad\\
 =  \sup_{\|f\|_{p\leq 1 }}  \lim_{n\to \infty}\|C_{\psi_n}f - C_{\psi}f\|_q= 0
 \end{align*} contradicting our assumption.

(b) Let $p= q$ and assume that $C_\psi \in C(\mathcal{F}_\varphi^p, \mathcal{F}_\varphi^p )$ is not compact. Then by Theorem~\ref{thm1}, $\psi(z)= az, |a|=1.$ We proceed to show that  $C_\psi$ is isolated. That is there exists a positive number c such that
\begin{align}
\label{isolated}
\|C_{\psi}- C_{\psi_1}\|\geq c
\end{align} for all  $C_{\psi_1} \in C(\mathcal{F}_\varphi^p, \mathcal{F}_\varphi^p )$ for which $\psi_1\neq \psi$.
We may first consider the forms   $\psi_1(z)= a_1z, \ \ |a_1|= 1$ and $a_1\neq a$. Since the polynomials are contained in $\mathcal{F}_\varphi^p$,
\begin{align}
\label{weakk}
\|C_{\psi}- C_{\psi_1}\| \geq  \sup_{n\geq 0}\|z^n\|_{p}^{-1}\big\| (C_\psi- C_{\psi_1})z^n \big\|_p \quad \quad  \quad \quad  \quad \quad  \quad \quad  \nonumber\\
= \sup_{n\geq 0}\|z^n\|_{p}^{-1}|a^n-a_1^n|\|z^n\|_{p}= \sup_{n\geq 0}|a^n-a_1^n| \geq 2.
\end{align}
On the other hand, if   $C_{\psi_1}$ is compact, then  $\psi_1= a_1z+b, \ \ \ |a_1| <1$ and using the unit norm sequence of functions  $f_{(w, R)}^*$ in \eqref{unitt}
\begin{align}
\label{weak}
\|C_{\psi}- C_{\psi_1}\| \geq   \sup_{w\in \CC} \big\| (C_\psi- C_{\psi_1})f_{(w, R)}^*  \big\|_{p}
\geq \sup_{w\in \CC}\bigg( \|C_\psi f_{(w, R)}^*\|_p-  \Big\|C_{\psi_1} f_{(w, R)}^*\Big\|_p\bigg)\nonumber\\
\gtrsim \sup_{w\in\CC}\bigg( 1- \Big\|C_{\psi_1} f_{(w, R)}^*\Big\|_{p}\bigg).
\end{align}
Now, $f_{(w, R)}^* \to 0$  weakly  as $|w|\to \infty$,  and as  $C_{\psi_1}$ is compact, we have
\begin{align*}
\Big\|C_{\psi_1} f_{(w, R)}^*\Big\|_p \to 0
\end{align*} as $|w|\to \infty$. This together  with \eqref{weak} for sufficiently big  $|w|$ gives
\begin{align}
\label{weaak}
\|C_{\psi}- C_{\psi_1}\| \gtrsim    1.
\end{align}
 From \eqref{weaak} and \eqref{weakk}, the  claim  in \eqref{isolated} follows.

(ii) If both operators are compact,  obviously the difference is also compact. Thus,  we shall prove the other implication, i.e. assuming  the difference is compact, we  need to verify that both composition operators are compact. We plan to argue in the direction of contradiction again,  and assume that one of them $C_{\psi_1}$  is not compact. It follows that $C_{\psi_2}$ is not compact either since for any $f\in \mathcal{F}_\varphi^p$
\begin{align*}
|C_{\psi_1}f(z)|^p  \lesssim  | (C_{\psi_1}-C_{\psi_2}) f(z)|^p+ |C_{\psi_2} f(z)|^p.
\end{align*}
Thus, we may set $\psi_1(z)= a_1z$ and  $\psi_2(z)= a_2z$ where $ a_1 \neq a_2$ and $|a_j|= 1, j= 1, 2$.  Since the unit norm sequence  $f_{(w, R)}^*$ is weakly convergent, compactness of the difference operator implies
\begin{align}
\label{compdiff}
\|(C_{\psi_1}-C_{\psi_2}) f_{(w, R)}^*\|_p \to 0 \ \ \text{as}\ \  |w| \to \infty.
\end{align}
On the other hand, we have a lower estimate
\begin{align*}
\|(C_{\psi_1}-C_{\psi_2}) f_{(w, R)}^* \|_p^p=\int_{\CC} |C_{\psi_1}f_{(w, R)}^*(z)-C_{\psi_2}f_{(w, R)}^*(z) |^p e^{-p\varphi(z)} dA(z)\nonumber\\
\geq \int_{D(z_0, \tau(z_0))} |C_{\psi_1}f_{(w, R)}^*(z)-C_{\psi_2}f_{(w, R)}^*(z) |^p e^{-p\varphi(z)} dA(z).
\end{align*} From  this  and applying  \eqref{pointwise} and \eqref{test} we estimate
 \begin{align*}
\|C_{\psi_1}-C_{\psi_2} f_{(w, R)}^*\|_p\gtrsim \tau(z_0)^{\frac{p}{2}}  |C_{\psi_1}f_{(w, R)}^*(z_0)-C_{\psi_2}f_{(w, R)}^*(z_0) | e^{-\varphi(z_0)} dA(z)\\
\simeq \frac{\tau(z_0)^{\frac{p}{2}}}{\tau(w)^{\frac{p}{2}} }  |f_{(w, R)}(\psi_1(z_0))-C_{\psi_2}f_{(w, R)}(\psi_2(z_0)) | e^{-\varphi(z_0)}.
\end{align*}
Setting $w= \psi_1(z_0)$ on the right-hand side above, applying  \eqref{test00} and \eqref{test0} and observing that
$\tau(\psi_1(z_0))= \tau(\psi_2(z_0))= \tau(z_0)$ leads to
\begin{align*}
\|(C_{\psi_1}-C_{\psi_2}) f_{(w, R)}^* \|_p\gtrsim \frac{\tau(z_0)^{\frac{p}{2}}}{\tau(\psi_1(z_0))^{\frac{p}{2}}} |f_{(\psi_1(z_0), R)}(\psi_1(z_0))-f_{(\psi_1(z_0), R)}(\psi_2(z_0)) | e^{-\varphi(z_0)}\nonumber\\
\geq \Big(|f_{(\psi_1(z_0), R)}(\psi_1(z_0))|-|f_{(\psi_1(z_0), R)}(\psi_2(z_0)) |\Big) e^{-\varphi(z_0)}\\
\gtrsim \bigg(e^{\varphi(z_0)}- e^{\varphi(z_0)} \bigg(\frac{\tau(z_0)}{|z_0||a_1-a_2|}\bigg)^{\frac{R^2}{2}}\bigg) e^{-\varphi(z_0)}= 1-\bigg(\frac{\tau(z_0)}{|z_0||a_1-a_2|}\bigg)^{\frac{R^2}{2}} = 1 \nonumber
\end{align*}  when $|z_0| \to \infty$ which  contradicts the fact in \eqref{compdiff}.

The statement in part (b) is an immediate consequence of part (a) and part (ii) of Theorem~\ref{thm2}.
\subsection{Proof of Theorem~\ref{thm5}}\label{essential}
Since
 the essential norm topology is weaker than the operator norm topology, each essentially isolated point is isolated.  Thus, we consider
     an operator $C_{\psi_1}\in  C(\mathcal{F}_p, \mathcal{F}_p)$, and assume that it is    isolated in the operator norm topology. Then  we plan to show that it is also  essentially isolated.  We may  let
 $\psi_1(z)= a_1z$ with $|a_1|= 1$ . It suffices to show that  for all bounded composition operators $C_{\psi_2} \in  C(\mathcal{F}_p, \mathcal{F}_p)$, the estimate
\begin{align*}
\|C_{\psi_1}-C_{\psi_2}\|_e \gtrsim 1
\end{align*} holds.  If $\psi_2$ is not compact either, then we  may set  $\psi_2(z)= a_2z$ where $ a_1 \neq a_2$ and $|a_2|= 1$. Then for  any compact operator $Q $ on $\mathcal{F}_p$ we have
 \begin{align*}
 \| (C_{\psi_1}- C_{\psi_2})-Q\| \geq \limsup_{| w| \rightarrow \infty } \|((C_{\psi_1}- C_{\psi_2})-Q) f_{(w, R)}^*\|_p \quad \quad \quad \quad \quad \quad\quad \quad \quad\\
  \geq \limsup_{\vert w\vert \rightarrow \infty }  \Vert (C_{\psi_1}- C_{\psi_2}) f_{(w, R)}^*\Vert_p-\Vert Q f_{(w, R)}^*\Vert_p = \limsup_{\vert w\vert \rightarrow \infty } \Vert (C_{\psi_1}- C_{\psi_2}) f_{(w, R)}^*\Vert_p.
 \end{align*}
Arguing as in the preceding  proof and setting $w= \psi_1(z_0)$ we find ,
\begin{align*}
\|C_{\psi_1}-C_{\psi_2} \|_e\gtrsim \limsup_{\vert z_0\vert \rightarrow \infty } \Big(|f_{(\psi_1(z_0), R)}(\psi_1(z_0))|-|f_{(\psi_1(z_0), R)}(\psi_2(z_0)) |\Big) e^{-\varphi(z_0)}\\
\gtrsim \limsup_{\vert z_0\vert \rightarrow \infty } \bigg(1-\bigg(\frac{\tau(z_0)}{|z_0||a_1-a_2|}\bigg)^{\frac{R^2}{2}}\bigg) = 1.
\end{align*}
On the other hand,  if $C_{\psi_2}$ is compact, we set  $\psi_2(z)= a_2z+b$ with $|a_2| <1$, and repeating the preceding  arguments
\begin{align*}
\|C_{\psi_1}-C_{\psi_2} \|_e\gtrsim \limsup_{\vert z_0\vert \rightarrow \infty } \Big(|f_{(\psi_1(z_0), R)}(\psi_1(z_0))|-|f_{(\psi_1(z_0), R)}(\psi_2(z_0)) |\Big) e^{-\varphi(z_0)}\\
\gtrsim  \limsup_{\vert z_0\vert \rightarrow \infty }\bigg(1-\bigg(\frac{\min\{\tau(z_0), \tau(a_2z_0+b_2)\}}{|z_0 (a_1-a_2)+b_2|}\bigg)^{\frac{R^2}{2}}\bigg) =1,
\end{align*}  and completes the proof.

\end{document}